\documentclass[10pt,reqno]{amsart}
\usepackage[left=2cm,right=2cm,top=2cm,bottom=2cm]{geometry}
\usepackage[colorlinks=true,allcolors=blue]{hyperref}
\usepackage{color}
\usepackage{amsmath, amssymb, amsthm}
\usepackage{mathtools}
\usepackage[noabbrev,capitalize,nameinlink,nosort]{cleveref}
\usepackage{enumerate}
\usepackage[hang,flushmargin]{footmisc}

\usepackage{booktabs}
\usepackage{makecell}
\usepackage{pgfplots}
\usepackage{pgfplotstable}
\pgfplotsset{compat=1.18}

\usepackage[shortlabels]{enumitem}
\crefformat{enumi}{#2#1#3}
\crefrangeformat{enumi}{#3#1#4 to~#5#2#6}
\crefmultiformat{enumi}{#2#1#3}
{ and~#2#1#3}{, #2#1#3}{ and~#2#1#3}

\let\originalleft\left
\let\originalright\right
\renewcommand{\left}{\mathopen{}\mathclose\bgroup\originalleft}
\renewcommand{\right}{\aftergroup\egroup\originalright}

\crefformat{equation}{#2(#1)#3}

\usepackage{aliascnt}
\newtheorem{lemma}{Lemma}[section]
\crefname{lemma}{Lemma}{Lemmas}
\Crefname{lemma}{Lemma}{Lemmas}

\newaliascnt{proposition}{lemma}

\aliascntresetthe{proposition}
\crefname{proposition}{Proposition}{Propositions}
\Crefname{proposition}{Proposition}{Propositions}

\newaliascnt{theorem}{lemma}
\newtheorem{theorem}[theorem]{Theorem}
\aliascntresetthe{theorem}
\crefname{theorem}{Theorem}{Theorems}
\Crefname{theorem}{Theorem}{Theorems}

\newaliascnt{corollary}{lemma}

\aliascntresetthe{corollary}
\crefname{corollary}{Corollary}{Corollaries}
\Crefname{corollary}{Corollary}{Corollaries}

\newaliascnt{fact}{lemma}
\newtheorem{fact}[fact]{Fact}
\aliascntresetthe{fact}
\crefname{fact}{Fact}{Facts}
\Crefname{fact}{Fact}{Facts}

\newaliascnt{conjecture}{lemma}

\aliascntresetthe{conjecture}
\crefname{conjecture}{Conjecture}{Conjectures}
\Crefname{conjecture}{Conjecture}{Conjectures}

\newaliascnt{claim}{lemma}

\aliascntresetthe{claim}
\crefname{claim}{Claim}{Claims}
\Crefname{claim}{Claim}{Claims}

\theoremstyle{definition}

\newaliascnt{definition}{lemma}

\aliascntresetthe{definition}
\crefname{definition}{Definition}{Definitions}
\Crefname{definition}{Definition}{Definitions}

\newaliascnt{example}{lemma}

\aliascntresetthe{example}
\crefname{example}{Example}{Examples}
\Crefname{example}{Example}{Examples}

\theoremstyle{remark}

\newaliascnt{remark}{lemma}
\newtheorem{remark}[remark]{Remark}
\aliascntresetthe{remark}
\crefname{remark}{Remark}{Remarks}
\Crefname{remark}{Remark}{Remarks}

\newtheorem*{remark*}{Remark}

\newcommand*{\claimproofname}{Proof of claim}

\crefname{section}{Section}{Sections}
\Crefname{section}{Section}{Sections}
\crefname{subsection}{Section}{Sections}
\Crefname{subsection}{Section}{Sections}
\crefname{equation}{}{} % keep equations as just numbers
\crefname{algorithm}{Algorithm}{Algorithms}
\Crefname{algorithm}{Algorithm}{Algorithms}

\AddToHook{cmd/appendix/after}{%
  \crefalias{section}{appendix}%
}

\newcommand{\floor}[1]{\left\lfloor #1 \right\rfloor}

\newcommand{\mc}{\mathcal}

\newcommand{\eps}{\varepsilon}

\newcommand{\hide}[1]{}

\allowdisplaybreaks
\usepackage{parskip}
\makeatletter
\def\thm@space@setup{%
  \thm@preskip=\parskip \thm@postskip=0pt
}
\def\th@remark{%
  \thm@space@setup
  \thm@headfont{\itshape}%
  \normalfont
}
\makeatother

\let\le\leqslant
\let\ge\geqslant

\newcommand{\RR}{\mathbb{R}}

\newcommand{\CC}{\mathbb{C}}

\newcommand{\PP}{\mathbb{P}}

\newcommand{\EE}{\mathbb{E}}
\DeclareMathOperator{\Inf}{Inf}

\DeclareMathOperator{\rank}{rank}
\DeclareMathOperator{\tr}{tr}
\DeclareMathOperator{\Span}{span}
\title[]{O\lowercase{ptimal bound for the polynomial} L\lowercase{ittlewood}--O\lowercase{fford problem}}

\author[]{A\lowercase{lexandr }G\lowercase{rebennikov}}
\address{Institute of Science and Technology Austria (ISTA)}
\email{aleksandr.grebennikov@ist.ac.at}

\begin{document} 
\begin{abstract}    
    We present an exposition of an argument, discovered by GPT-6 Pro, that gives an optimal bound for the polynomial Littlewood--Offord problem. Namely, let $F$ be a degree-$d$ multilinear polynomial that contains $r$ degree-$d$ monomials involving disjoint sets of variables. Then, for i.i.d.\ Rademacher random variables $\xi_1, \ldots, \xi_n$, we have $\PP[F(\xi_1, \ldots, \xi_n) = 0] = O_d(r^{-1/2})$. This improves upon the previous bound of $(\log r)^{O_d(1)} r^{-1/2}$ due to Meka, O.\ Nguyen, and Vu, and resolves a conjecture attributed to H.\ Nguyen and Vu. The key part of the proof is an estimate for the total influence of bounded-degree rational functions, which resolves a recent conjecture of Kothari, Kovacs-Deak, Wang, and Yang.  
\end{abstract}

\maketitle

\textbf{Statement of AI use.} The proof was discovered autonomously by GPT-6 Pro (the chat logs can be found \href{https://chatgpt.com/share/6ab4f17d-7a3c-83eb-a4e6-a99920bd9384}{here}), and the author does not claim any credit for the proof ideas. The argument was subsequently verified, digested, and rewritten from scratch by the author.

\section{Introduction}

The story of polynomial anticoncentration begins with the following question, asked by Littlewood and Offord~\cite{littlewood-offord} in connection with their study of random algebraic equations. If $c_1, \ldots, c_n$ are nonzero real numbers, and $\xi_1, \ldots, \xi_n$ are i.i.d.\ Rademacher random variables (i.e.\ $\PP[\xi_i = 1] = \PP[\xi_i = -1] = 1/2$), what is the maximum probability that the random sum $\sum_{i=1}^n c_i \xi_i$ equals a given value? Littlewood and Offord obtained a bound of the form $O(\log n/\sqrt{n})$ on this probability; shortly after, Erd\H{o}s~\cite{erdos-littlewood-offord} gave an elegant combinatorial proof of the following precise result:
\begin{equation} \label{eq:ELO}
\sup_{x \in \RR}\PP\left[\sum_{i=1}^n c_i \xi_i = x\right] \le {2^{-n}}\binom{n}{\lfloor n/2 \rfloor} = O\left(1/\sqrt{n}\right)\,.
\end{equation}
The polynomial Littlewood--Offord problem asks for an analogue of this bound for polynomials of any fixed degree. This direction was studied by Rosi\'nski and Samorodnitsky~\cite{rosinski-samorodnitsky} in the context of L\'evy chaos, by Costello, Tao, and Vu~\cite{costello-tao-vu} in their work on the singularity of random symmetric matrices, and by Razborov and Viola~\cite{razborov-viola} in their work on polynomial approximations of Boolean functions; for a more detailed history of this problem, we refer to the survey of H.\ Nguyen and Vu~\cite{nguyen-survey} and to the introduction of \cite{kwan-sauermann}.

For a polynomial $F(x_1, \ldots, x_n)$ of degree $d \ge 2$, one cannot obtain a bound on the vanishing probability of $F(\xi_1, \ldots, \xi_n)$ analogous to \cref{eq:ELO} without some additional assumption ensuring that $F$ is ``far from the zero polynomial'': for instance, if $F$ has a factor $(x_1 - 1)$, then $\PP[F(\xi_1, \ldots, \xi_n) = 0] \ge 1/2$. One way to rule out such examples is to assume that $F$ has many nonzero coefficients. H.\ Nguyen and Vu (see \cite{meka-nguyen-vu,razborov-viola}) conjectured that if $F$ is a degree-$d$ multilinear\footnote{A polynomial is multilinear if it is linear in each variable when the others are held fixed; since $\xi_i^2 = 1$, the question about arbitrary polynomials can be essentially reduced to the multilinear case.} polynomial with $m n^{d-1}$ nonzero degree-$d$ coefficients, then 
\begin{equation} \label{eq:nguyen-vu}
\PP[F(\xi_1, \ldots, \xi_n) = 0] = O_d\left(1/\sqrt{m}\right)\,.
\end{equation}
Previously, this bound was only known in the following special cases: $d = 1$, by the classical theorem of Erd\H{o}s discussed above; $d = 2$, by a result of Kwan and Sauermann~\cite{kwan-sauermann}; and when the polynomial $F$ has bounded Chow rank, by a result of Grebennikov and Kwan~\cite[Theorem 1.1(1)]{grebennikov-kwan-littlewood-offord}.
For a general fixed degree $d$, Costello, Tao, and Vu~\cite{costello-tao-vu} and Razborov and Viola~\cite{razborov-viola} obtained bounds on this probability of the form $m^{-c_d}$ for some constant $c_d > 0$. Meka, O.\ Nguyen, and Vu~\cite[Theorem 1.6]{meka-nguyen-vu} improved the bound to $(\log m)^{O_d(1)}/\sqrt{m}$ using a result of Kane~\cite{kane}.

In this work, we establish the bound \cref{eq:nguyen-vu}, resolving the Nguyen--Vu conjecture. In fact, instead of requiring $F$ to have many nonzero coefficients, we find it more convenient to work under a weaker assumption that $F$ has a large ``matching of degree-$d$ monomials'' (which also appeared in \cite{meka-nguyen-vu,razborov-viola}). Namely, for a degree-$d$ multilinear polynomial $F$ of the form
\begin{equation} \label{eq:multilinear-polynomial}
F(x_1, \ldots, x_n) = \sum_{S \subseteq [n], |S| \le d} a_S x^{S}\,, \qquad \text{ where }\qquad x^S = \prod_{i \in S} x_i\,,
\end{equation}
define $r(F)$ as the maximum $r$ such that there exist pairwise disjoint sets $S_1, \ldots, S_r \subseteq [n]$ of size $d$ such that $a_{S_i} \neq 0$ for all $i \in [r]$.

\begin{theorem}[Pointwise anticoncentration] \label{thm:pointwise-anticoncentration}
Let $F$ be a multilinear polynomial of degree $d \ge 1$ with complex coefficients. Then, for some absolute constant $C$,
\begin{equation*}
    \PP\big[F(\xi_1, \ldots, \xi_n) = 0\big] \le \frac{C 2^d \sqrt{d}}{\sqrt{r(F)}}\,.
\end{equation*}
\end{theorem}
Note that if $F$ has $m n^{d-1}$ nonzero degree-$d$ coefficients, then one can greedily find $r = \lceil m/d\rceil$ disjoint sets $S_1, \ldots, S_{r}$ of size $d$ such that $a_{S_i} \neq 0$ for each $i \in [r]$, which means that $r(F) \ge m/d$. Thus, \cref{thm:pointwise-anticoncentration} indeed implies \cref{eq:nguyen-vu}. 

More generally, for $\gamma > 0$ and a degree-$d$ multilinear polynomial $F$ as in \cref{eq:multilinear-polynomial}, define $r_{\gamma}(F)$ as the maximum $r$ such that there exist pairwise disjoint sets $S_1, \ldots, S_r \subseteq [n]$ of size $d$ such that $|a_{S_i}| \ge \gamma$ for all $i \in [r]$. 
\cref{thm:pointwise-anticoncentration} follows from the slightly stronger \cref{thm:small-ball-anticoncentration} by taking $\gamma \to 0$. 

\begin{theorem}[Small-ball anticoncentration] \label{thm:small-ball-anticoncentration}
Let $F$ be a multilinear polynomial of degree $d \ge 1$ with complex coefficients. Then, for every $\gamma > 0$ and some absolute constant $C$,
\begin{equation*}
    \PP\big[|F(\xi_1, \ldots, \xi_n)| \le \gamma\big] \le \frac{C 2^d \sqrt{d}}{\sqrt{r_{\gamma}(F)}}\,.
\end{equation*}
\end{theorem}

\begin{remark}
Naturally, one can ask about analogues of \cref{thm:pointwise-anticoncentration,thm:small-ball-anticoncentration} for polynomials of i.i.d.\ random variables that are not Rademacher. When these random variables are ``not too close to being constant'', one can reduce this question to the Rademacher case via a conditioning argument analogous to \cite[Section 11]{kwan-sauermann}.
\end{remark}

Our next result, \cref{thm:pointwise-anticoncentration-robust}, is a different extension of \cref{thm:pointwise-anticoncentration} (inspired by \cite{kwan-sauermann}) that replaces $r(F)$ with a different parameter $s(F)$: the number of variables on which $F$ ``robustly depends''. Namely, for a multilinear polynomial $F$ in $n$ variables, define $s(F)$ as the maximum integer $s \le n+1$ such that for every set $I \subseteq [n]$ of size less than $s$ and every possible assignment of $\pm 1$ values to the variables $(x_i)_{i \in I}$, the resulting polynomial in the remaining variables is not identically zero. \cref{thm:pointwise-anticoncentration-robust} follows from \cref{thm:pointwise-anticoncentration} by a simple inductive argument, but might be more convenient to use when one has additional information about the polynomial beyond its top-degree coefficients. 

\begin{theorem} \label{thm:pointwise-anticoncentration-robust}
Let $F$ be a multilinear polynomial of degree $d \ge 1$ with complex coefficients. Then, for some absolute constant $C_1$,
\begin{equation*}
    \PP\big[F(\xi_1, \ldots, \xi_n) = 0\big] \le \frac{C_1 2^d d}{\sqrt{s(F)}}\,.
\end{equation*}
\end{theorem}

As observed in \cite{meka-nguyen-vu}, the exponential dependence on $d$ in \cref{thm:pointwise-anticoncentration,thm:small-ball-anticoncentration,thm:pointwise-anticoncentration-robust} is unavoidable. Namely, one can check that if $2^d \mid r$, then
$$
\PP\left[\left(\sum_{i=1}^r \prod_{j=1}^d (1 + \xi_{i, j})\right) - r = 0\right] = \Theta\left(\frac{2^{d/2}}{\sqrt{r}}\right)\,.
$$
However, the specific dependence on $d$ in these theorems is likely non-optimal. 

\begin{remark}
The polynomial Littlewood--Offord problem has found a number of applications to other questions in combinatorics and computer science, and hence our results have immediate implications for some of them. We briefly mention two such applications:
\begin{itemize}[leftmargin=2em]
    \item \emph{Polynomial approximations of the parity function.} Razborov and Viola~\cite{razborov-viola} proposed a ``challenge'' to construct an \emph{explicit} Boolean function that cannot be approximated by polynomials of degree at most $\log_2 n$. They suggested the parity function as a candidate, and proved that $\PP[F(\xi_1, \ldots, \xi_n) = \prod_{i=1}^n \xi_i] \le 1/2$ for every polynomial $F$ of degree at most $\frac{1}{2} \log_2 \log_2 n$ and sufficiently large $n$. Meka, O.\ Nguyen, and Vu~\cite{meka-nguyen-vu} extended this result to polynomials of degree at most $\frac{\log n}{15 \log \log n}$. Plugging the bound given by \cref{thm:pointwise-anticoncentration} into the approach of Razborov and Viola extends it further to polynomials of degree at most $c \log_2 n$ for any fixed $c < 1/2$: indeed, if $r(F) \ge n/\log_2 n$ then the desired bound follows from \cref{thm:pointwise-anticoncentration} applied to $F-1$ and $F+1$; otherwise, we can condition on the outcomes of $n/2$ variables to reduce the degree of our polynomial and iterate the argument.
    \item \emph{Edge-statistics in large hypergraphs.} The connection between this topic and the polynomial Littlewood--Offord problem was first observed by Kwan, Sudakov, and Tran~\cite{kwan-sudakov-tran}, and later used in \cite{grebennikov-kwan-edge-statistics,jain-kwan-mubayi-tran,kwan-sauermann}. Using this connection, Jain, Kwan, Mubayi, and Tran~\cite{jain-kwan-mubayi-tran} proved that the fraction of $k$-vertex subsets in a large $r$-uniform hypergraph that induce exactly $\ell \le \frac{1}{2}\binom{k}{r}$ edges is at most $(\ell/k^{r-1})^{-1/2 + o(1)}$. Our results remove the ``$o(1)$-loss'', yielding the optimal bound of $O_r((\ell/k^{r-1})^{-1/2})$.
\end{itemize}
\end{remark}

Another research direction, initiated by Costello~\cite{costello} and further developed in \cite{grebennikov-kwan-littlewood-offord,jin-kwan-sauermann-wang}, is to study the conditions on the polynomial that force \emph{better} anticoncentration than predicted by the Nguyen--Vu conjecture. For instance, Jin, Kwan, Sauermann, and Wang conjectured~\cite[Conjecture 1.7]{jin-kwan-sauermann-wang} that if a degree-$d$ polynomial is ``robustly irreducible'' (i.e., cannot be made reducible by modifying $o(n^d)$ of its coefficients), then one can get a power-saving improvement over \cref{eq:nguyen-vu}. It is not clear to us whether our approach can be used to make progress on this or other related questions.

\subsection{Total influence of bounded-degree rational functions} For a function $f: \{-1, 1\}^n \to \mathbb{R}$, its \emph{total influence} is defined as
\begin{equation*}
    \Inf[f] = \sum_{i=1}^n \EE\big[(f(\xi) - f(\xi^{(i)}))^2\big]\,,
\end{equation*}
where $\xi = (\xi_1, \ldots, \xi_n)$ is a uniformly random vector in $\{-1, 1\}^n$, and $\xi^{(i)} = (\xi_1, \ldots, \xi_{i-1}, -\xi_i, \xi_{i+1}, \ldots, \xi_n)$.

As in the work of Meka, O.\ Nguyen, and Vu \cite{meka-nguyen-vu}, our proof proceeds by establishing a bound on the total influence of certain functions on the Boolean cube. However, the class of functions we consider is different.
In \cite{meka-nguyen-vu}, the authors work with polynomial threshold functions, i.e.\ functions of the form
$$\mathbf{1}_{p > 0}(x) = \begin{cases} 1, & \text{if } p(x) > 0; \\ 0, & \text{if } p(x) \le 0; \end{cases} \qquad \text{ for each $x = (x_1, \ldots, x_n) \in \{-1, 1\}^n$}$$
for a real polynomial $p$. The Gotsman--Linial conjecture states that\footnote{In their original paper \cite{gotsman-linial}, Gotsman and Linial conjectured a more precise bound, which was later disproved
\cite{chapman,kim-maldonado-wellens}.} 
\begin{equation} \label{eq:gotsman-linial}
\Inf[\mathbf{1}_{p > 0}] = O(\deg(p) \sqrt{n}),
\end{equation}
and the authors of \cite{meka-nguyen-vu} used a bound towards this conjecture due to Kane \cite{kane} to derive their anticoncentration result.

We instead work with rational functions. In short, their relevance to the problem can be explained as follows: for the polynomial Littlewood--Offord problem, it suffices to control the total influence of $\mathbf{1}_{p \neq 0}$ rather than $\mathbf{1}_{p > 0}$, and $\mathbf{1}_{p \neq 0}$ can be approximated by the rational functions $\tfrac{p}{p + \eps}$ as $\eps \to 0$. Our next theorem gives an upper bound on the total influence of rational functions.

\begin{theorem} \label{thm:influence}
Suppose that a function $f: \{-1, 1\}^n \to \mathbb{R}$ can be written as $f = p/q$, where $p$ and $q$ are real polynomials of degree at most $d$ in $n$ variables (in particular, $q(x) \neq 0$ for all $x \in \{-1, 1\}^n$). Then, 
\begin{equation*}
    \Inf[f] \le \left(4\sqrt{2\pi} + O(1/n)\right) d \sqrt{n} \|f\|^2_{\infty}\,,
\end{equation*}
where $\|f\|_{\infty} = \max_{x \in \{-1, 1\}^n} |f(x)|$.
\end{theorem}

\cref{thm:influence} does not appear to be strong enough to imply the Gotsman--Linial conjecture. However, we can deduce a weaker\footnote{To see that this statement is indeed weaker than the Gotsman--Linial conjecture, note that $\mathbf{1}_{p \neq 0} = \mathbf{1}_{p^2 > 0}$.} version of it proposed by Kothari, Kovacs-Deak, Wang, and Yang \cite[Conjecture 2]{KKWY}, which states that $\Inf[\mathbf{1}_{p \neq 0}] = O(\deg(p) \sqrt{n})$. 

\begin{theorem} \label{thm:influence-pointwise}
For a real polynomial $p$ in $n$ variables, we have
\begin{equation*}
    \Inf[\mathbf{1}_{p \neq 0}] \le \left(\sqrt{2\pi} + O(1/n)\right) \deg (p) \sqrt{n}.
\end{equation*}
\end{theorem}
Since the polynomial $p_0 = \prod_{i = 1}^d ((x_1 + \cdots + x_n) - 4i)$ satisfies $\Inf[\mathbf{1}_{p_0 \neq 0}] = (2\sqrt{2/\pi} + O_d(1/n)) d \sqrt{n}$ for every fixed $d$ and sufficiently large even $n$, the bound given by \cref{thm:influence-pointwise} is sharp up to a factor of $\pi/2 + O_d(1/n)$.

\subsection{Proof ideas} While our arguments are quite short, we still include a brief overview of the main ideas. We focus on the proof of \cref{thm:influence}, since the other results follow from it in a fairly standard way. This proof has some high-level similarities to the arguments of Aaronson, Ben-David, Kothari, Rao, and Tal~\cite[Section 4.1]{ABKRT}, 
although the setting and the details are different.

Perhaps the main idea of the proof (which was also used in \cite{ABKRT}) is to study pointwise multiplication by $f$ as a linear map from the vector space of functions $\{-1, 1\}^n \to \mathbb{R}$ to itself. The matrix $A$ of this map in the Fourier basis $(x^S)_{S \subseteq [n]}$ has entries $A_{S_1, S_2} = \hat{f}(S_1 \triangle S_2)$. Hence, recalling the standard identity $\Inf[f] = 4 \sum_{S \subseteq [n]} |S| \hat{f}(S)^2$, we need to bound a certain weighted sum of squares of the entries of $A$.

Using basic linear algebra, we can express the sum of squares of the entries of (a submatrix of) $A$ as the sum of squares of its singular values, which can in turn be bounded in terms of its rank and operator norm. 
Clearly, the operator norm of $A$ is at most $\|f\|_{\infty}$. While $A$ itself may have high rank, we can still show that it contains large submatrices of low rank. To see this, note that $A = P Q^{-1}$, where $P$ and $Q$ are the matrices of the linear maps given by pointwise multiplication by $p$ and by $q$, respectively. Since $p$ and $q$ are polynomials of degree at most $d$, the matrices $P$ and $Q$ have many zero entries: $P_{S_1, S_2} = Q_{S_1, S_2} = 0$ whenever $|S_1 \triangle S_2| > d$. This allows us to control, for each $k$, the rank of the submatrix of $A$ whose rows are indexed by sets of size greater than $k$ and whose columns are indexed by sets of size at most $k$. 
Then, we take a sum over all $k$ and use a double-counting argument to complete the proof.

\textbf{Acknowledgements.} The author is grateful to Zhihan Jin for insightful discussions of the proof, and to Matthew Kwan for helpful comments on an earlier version of this exposition. The author is supported by ERC Starting Grant “RANDSTRUCT” No. 101076777.

\bigskip

\underline{\textbf{Update 07.10.2026.}} Shortly after the first version of this note was made public, OpenAI~\cite{openai2026-gotsman-linial} released a proof of the Gotsman--Linial conjecture as stated in \cref{eq:gotsman-linial}. Combined with (a slight refinement of) the reduction from \cite{meka-nguyen-vu}, this gives an alternative proof of our main anticoncentration results (\cref{thm:pointwise-anticoncentration,thm:small-ball-anticoncentration,thm:pointwise-anticoncentration-robust}).
The proof in \cite{openai2026-gotsman-linial} is based on ideas roughly similar to those presented here, but also introduces several additional ingredients.

\section{The key lemma}

In this section, we prove the key estimate on the Fourier coefficients of a bounded-degree rational function (\cref{lemma:fourier}), which is the main novelty of our approach. 

Throughout the paper, we write $\xi_1, \ldots, \xi_n$ for i.i.d.\ Rademacher random variables, and $\xi = (\xi_1, \ldots, \xi_n)$ for a uniformly random vector in $\{-1, 1\}^n$.
The monomials $(x^S)_{S \subseteq [n]}$ form an orthonormal basis of the space of functions $\{-1, 1\}^n \to \mathbb{R}$ with respect to the inner product $\langle f, g \rangle =  \EE[f(\xi) g(\xi)]$. Thus, every function $f: \{-1, 1\}^n \to \mathbb{R}$ can be written as
$$f(x) = \sum_{S \subseteq [n]} \hat{f}(S) x^S, \qquad \text{ where } \qquad \hat{f}(S) = \EE[f(\xi) \xi^S]\,.$$

For an integer $m \ge 0$, define $\lambda_m = \EE\max\{\xi_1 + \ldots + \xi_m, 0\}$. Using the fact that $\lambda_{m+1} - \lambda_m = \frac{1}{2}\PP[\xi_1 + \ldots + \xi_m = 0]$ and Stirling's formula, one can check that for $m \ge 1$,
\begin{equation} \label{eq:stirling}
\lambda_m = \frac{m}{2^m}\binom{m-1}{\floor{\frac{m-1}{2}}} = \left(\frac{1}{\sqrt{2\pi}} + O(1/m)\right) \sqrt{m}\,.\end{equation}

\begin{lemma} \label{lemma:fourier}
    Suppose that a function $f: \{-1, 1\}^n \to \mathbb{R}$ can be written as $f = p/q$, where $p$ and $q$ are real polynomials of degree at most $d$ in $n$ variables (in particular, $q(x) \neq 0$ for all $x \in \{-1, 1\}^n$). Then, 
    \begin{equation*}
    \sum_{S \subseteq [n]} \lambda_{|S|} \hat{f}(S)^2 \le d \cdot \|f\|^2_{\infty}\,.
    \end{equation*}
\end{lemma}
Our proof of \cref{lemma:fourier} has a linear-algebraic part and a combinatorial part. The linear-algebraic part is captured by the following lemma\footnote{As pointed out to us by Zhihan Jin, a similar linear-algebraic statement appears implicitly in the recent resolution of Chv\'atal's conjecture by Chang, H.\ Liu, and M.\ Liu~\cite{chang-liu-liu} and in the follow-up works \cite{ellis-filmus-friedgut,keevash}. We note that these results are independent of the current work and were also obtained with AI assistance.} that applies to arbitrary real-valued functions on the Boolean cube.

\begin{lemma} \label{lemma:linear-algebra}
    Let $\mc{L}$ be a family of subsets of $[n]$, and denote $V_{\mc{L}} = \Span_{\RR}\{x^S : S \in \mc{L}\}$. For a function $f: \{-1, 1\}^n \to \mathbb{R}$, also define the vector space $$K_{\mc{L}}(f) = \{h \in V_{\mc{L}} : fh \in V_{\mc{L}}\}\,.$$ Then,
    \begin{equation*}
    \sum_{\substack{S_1 \notin \mc{L} \\ S_2 \in \mc{L}}} \hat{f}(S_1 \triangle S_2)^2 \le \big(|\mc{L}| - \dim K_{\mc{L}}(f)\big) \cdot \|f\|^2_{\infty}\,.
    \end{equation*}
\end{lemma}
\begin{proof}
    Consider the linear map from the vector space of functions $\{-1, 1\}^n \to \mathbb{R}$ to itself given by $\varphi \mapsto f \varphi$.
    Let $A$ be the matrix of this map in the basis $(x^S)_{S \subseteq [n]}$ (with rows and columns indexed by subsets of $[n]$). Then, the entries of $A$ can be expressed in terms of the Fourier coefficients of $f$ as 
    \begin{equation*}
    A_{S_1, S_2} = \EE\left[f(\xi) \xi^{S_1} \cdot \xi^{S_2}\right] = \hat{f}(S_1 \triangle S_2)\,.
    \end{equation*}
    Denote $\mc{L}^c = \{S \subseteq [n]: S \notin \mc{L}\}$, and let $B = A[\mc{L}^c, \mc{L}]$ be the submatrix of $A$ with rows indexed by $\mc{L}^c$ and columns indexed by $\mc{L}$. Let $\sigma_1 \ge \ldots \ge \sigma_{\rank(B)} > 0$ be the nonzero singular values of $B$. Using that the largest singular value equals the operator norm $\|B\|$, we conclude that 
    \begin{equation} \label{eq:sum-of-squares}
    \sum_{\substack{S_1 \notin \mc{L} \\ S_2 \in \mc{L}}} \hat{f}(S_1 \triangle S_2)^2 = \sum_{\substack{S_1 \notin \mc{L} \\ S_2 \in \mc{L}}} B_{S_1, S_2}^2 = \tr(B^T B) = \sum_{i=1}^{\rank(B)} \sigma_i^2 \le \rank(B) \cdot \|B\|^2.
    \end{equation}
    It remains to bound the rank and the operator norm of $B$. Indeed, since $B$ is a submatrix of $A$, 
    \begin{equation} \label{eq:norm-of-B}
    \|B\| \le \|A\| = \sup_{\|\varphi\|_2 = 1} \|f \varphi\|_2 = \|f\|_{\infty}\,,
    \end{equation}
    and, since the kernel of $B$ can be identified with the set of functions $h \in V_{\mc{L}}$ such that $fh \in V_{\mc{L}}$,
    \begin{equation} \label{eq:rank-of-B}
    \rank(B) = |\mc{L}| - \dim \ker(B) = |\mc{L}| - \dim K_{\mc{L}}(f)\,.
    \end{equation}
    Substituting \cref{eq:norm-of-B,eq:rank-of-B} into \cref{eq:sum-of-squares} gives the desired bound. 
\end{proof}

Now, we combine \cref{lemma:linear-algebra} with a couple of combinatorial observations to prove \cref{lemma:fourier}.
\begin{proof}[\textbf{Proof of \cref{lemma:fourier}}]
    By rescaling, we may assume that $\|f\|_{\infty} = 1$. For an integer $k$, denote 
    $$\mc{L}_k = \{S \subseteq [n]: |S| \le k\}, \qquad V_k = \Span_{\RR}\left\{x^S : S \in \mc{L}_k\right\}, \qquad K_k(f) = \{h \in V_k : fh \in V_k\}\,.$$
    For each $0 \le k \le n-1$, \cref{lemma:linear-algebra} gives 
    \begin{equation}  \label{eq:apply-linear-algebra}
    \sum_{\substack{|S_1| > k \\ |S_2| \le k}} \hat{f}(S_1 \triangle S_2)^2 \le |\mc{L}_k| - \dim K_k(f).
    \end{equation}
    Clearly, $\dim V_k = |\mc{L}_k| = \sum_{\ell = 0}^k \binom{n}{\ell}$. To get a lower bound on the dimension of $K_k(f)$, we recall that $f = p/q$ where $p$ and $q$ are polynomials of degree at most $d$. Since $q(x) \neq 0$ for each $x \in \{-1, 1\}^n$, multiplication by $q$ is invertible, and thus the subspace $qV_{k-d} = \{qh : h \in V_{k-d}\} \subseteq V_k$ has the same dimension as $V_{k-d}$. Furthermore, 
    $$fqV_{k-d} = \{fqh : h \in V_{k-d}\} = \{ph : h \in V_{k-d}\} \subseteq V_k\,,$$
    and therefore 
    \begin{equation*}
    \dim K_k(f) \ge \dim (qV_{k-d}) = \dim V_{k-d} = \sum_{0 \le \ell \le k-d} \binom{n}{\ell}.
    \end{equation*}
    Substituting this into \cref{eq:apply-linear-algebra}, summing over $k = 0, 1, \ldots, n-1$, and dividing by $2^n$ yields
    \begin{equation} \label{eq:sum-over-k}
    \frac{1}{2^n}\sum_{k = 0}^{n-1} \sum_{\substack{|S_1| > k \\ |S_2| \le k}} \hat{f}(S_1 \triangle S_2)^2 \le \frac{1}{2^n}\sum_{k = 0}^{n-1} \sum_{\ell = \max(k-d+1,0)}^k \binom{n}{\ell} \le \frac{d}{2^n} \sum_{\ell = 0}^n \binom{n}{\ell} = d\,.
    \end{equation}
    By exchanging the order of summation, we can rewrite the left-hand side of \cref{eq:sum-over-k} as
    \begin{equation} \begin{aligned} \label{eq:exchange-sum}
    \frac{1}{2^n}\sum_{k = 0}^{n-1} \sum_{\substack{|S_1| > k \\ |S_2| \le k}} \hat{f}(S_1 \triangle S_2)^2 &= \sum_{S \subseteq [n]} \left(\frac{1}{2^n} \sum_{T \subseteq [n]} \max\Big\{|T| - |S \triangle T|, 0\Big\}\right) \hat{f}(S)^2\,.
    \end{aligned} \end{equation}
    The coefficient of $\hat{f}(S)^2$ in the last sum is the expected value of $\max\big\{|T| - |S \triangle T|, 0\big\}$ for a uniformly random $T \subseteq [n]$. Writing $T = \{i \in [n]: \xi_i = 1\}$, we have 
    \begin{equation*}
    \EE\max\Big\{|T| - |S \triangle T|, 0\Big\} = \EE\max\Big\{\sum_{i \in S} \xi_i, 0\Big\} = \lambda_{|S|}\,.
    \end{equation*}
    Combined with \cref{eq:sum-over-k} and \cref{eq:exchange-sum}, this gives the desired bound.
\end{proof}

\section{Proofs of the main results}

In this section, we deduce the results stated in the introduction from \cref{lemma:fourier}. These deductions are quite standard and similar to the arguments in \cite{meka-nguyen-vu}.

\begin{fact}[{see e.g.\ \cite[Theorem 2.38]{odonnell-book}}]\label{fact:influence-fourier}
    For every function $f: \{-1, 1\}^n \to \mathbb{R}$, we have
    \begin{equation*}
    \Inf[f] = 4\sum_{S \subseteq [n]} |S| \hat{f}(S)^2.
    \end{equation*}
\end{fact}
\begin{proof}
    For $i \in [n]$, let $f_{(i)}$ denote the function given by $$f_{(i)}(x_1, \ldots, x_n) = f(x_1, \ldots, x_n) - f(x_1, \ldots, x_{i-1}, -x_i, x_{i+1}, \ldots, x_n)\,.$$ Then, for each $S \subseteq [n]$, 
    $$ \hat{f}_{(i)}(S) = \begin{cases} 2\hat{f}(S), & \text{if } i \in S; \\ 0, & \text{if } i \notin S. \end{cases}$$
    Thus, by Parseval's identity, 
    \begin{equation*}
    \Inf[f] = \sum_{i=1}^n \EE\big[f^2_{(i)}(\xi)\big] = \sum_{i=1}^n \sum_{S \subseteq [n]} \hat{f}_{(i)}(S)^2 = 4\sum_{i=1}^n \sum_{\substack{S \subseteq [n] \\ S \ni i}} \hat{f}(S)^2 = 4\sum_{S \subseteq [n]} |S| \hat{f}(S)^2. \qedhere
    \end{equation*}
\end{proof}

\begin{proof}[\textbf{Proof of \cref{thm:influence}}]
    By \cref{eq:stirling}, for each $1 \le m \le n$ we have 
    \begin{equation*}
    \frac{m}{\lambda_m} \le \frac{n}{\lambda_n} = \left(\sqrt{2\pi} + O(1/n)\right) \sqrt{n}\,.
    \end{equation*}
    Therefore, by \cref{fact:influence-fourier} and \cref{lemma:fourier},
    \begin{equation*}
    \Inf[f] = 4\sum_{S \subseteq [n]} |S| \hat{f}(S)^2 \le \left(4\sqrt{2\pi} + O(1/n)\right) \sqrt{n} \sum_{S \subseteq [n]} \lambda_{|S|} \hat{f}(S)^2 \le \left(4\sqrt{2\pi} + O(1/n)\right) d \sqrt{n}\|f\|^2_{\infty}. \qedhere
    \end{equation*}
\end{proof}

\begin{proof}[\textbf{Proof of \cref{thm:influence-pointwise}}] 
    Denote $\delta = \min\{|p(x)|: x \in \{-1, 1\}^n,\,\; p(x) \neq 0\}$ (and set $\delta = \infty$ if this set is empty). For each $\eps \in (0, \delta)$, define the function $f_{\eps}:\{-1, 1\}^n \to \RR$ as
    $$f_{\eps}(x) = \frac{p(x)}{p(x) + \eps} \qquad \text{for all } x \in \{-1, 1\}^n\,.$$
    Then, $\lim\limits_{\eps \to 0} f_{\eps}(x) = \mathbf{1}_{p\neq 0}(x)$ for each $x \in \{-1, 1\}^n$, and thus $\lim\limits_{\eps \to 0}\|f_{\eps} - \tfrac{1}{2}\|_{\infty} = \tfrac{1}{2}$.
    Applying \cref{thm:influence} to each of the functions $f_{\eps} - \tfrac{1}{2}$ and exchanging the finite sum and the limit, we conclude that
    \begin{align*}
    \Inf[\mathbf{1}_{p \neq 0}] = 
    \Inf[\mathbf{1}_{p \neq 0} - \tfrac{1}{2}] = \lim_{\eps \to 0} \Inf[f_{\eps} - \tfrac{1}{2}] &\le \left(4\sqrt{2\pi} + O(1/n)\right) \deg (p) \sqrt{n} \cdot \lim_{\eps \to 0} \|f_{\eps} - \tfrac{1}{2}\|^2_{\infty} \\
    &= \left(\sqrt{2\pi} + O(1/n)\right) \deg (p) \sqrt{n}\,. \qedhere
    \end{align*}
\end{proof}

The next lemma is a variant of \cite[Eq.\ 3.1]{meka-nguyen-vu}, which connects the anticoncentration of a polynomial to the total influence of a certain rational function. 

\begin{lemma} \label{lemma:anticoncentration-influence}
    Let $G(x_1, \ldots, x_d) = \sum_{S \subseteq [d]} a_S x^S$ be a multilinear polynomial of degree $d$ in $d$ variables with complex coefficients such that $|a_{[d]}| \ge 1$. Also, let $g: \{-1, 1\}^d \to \RR$ be the function given by $g(x) = \frac{1}{1 + |G(x)|^2}$. Then,
    \begin{equation*}
    \PP\big[|G(\xi_1, \ldots, \xi_d)| \le \tfrac{1}{2}\big] \le \frac{64 \cdot 2^d}{d+1} \Inf[g]\,.
    \end{equation*}
\end{lemma}
\begin{proof}
    If $\PP\left[|G(\xi_1, \ldots, \xi_d)| \le \tfrac{1}{2}\right] = 0$, then we are done. Otherwise, there is a point $x_- \in \{-1, 1\}^d$ such that $|G(x_-)| \le \tfrac{1}{2}$. In this case, we will show that 
    \begin{equation} \label{eq:sb-i-0}
    \frac{64 \cdot 2^d}{d+1} \Inf[g] \ge 1 \ge \PP\big[|G(\xi_1, \ldots, \xi_d)| \le \tfrac{1}{2}\big].
    \end{equation}
    Since 
    \begin{equation*}
    \max_{x \in \{-1, 1\}^d} |G(x)| \ge \left|\EE\big[G(\xi_1, \ldots, \xi_d) \xi_1 \cdots \xi_d\big]\right| = |a_{[d]}| \ge 1\,,
    \end{equation*}
    there also exists a point $x_+ \in \{-1, 1\}^d$ such that $|G(x_+)| \ge 1$. So, 
    \begin{equation} \label{eq:sb-i-1}
    |g(x_-) - g(x_+)| = \left|\frac{1}{1 + |G(x_-)|^2} - \frac{1}{1 + |G(x_+)|^2}\right| \ge \frac{4}{5} - \frac{1}{2} > \frac{1}{4}.
    \end{equation}
    On the other hand, by Cauchy--Schwarz and \cref{fact:influence-fourier},
    \begin{equation}\begin{aligned} \label{eq:sb-i-2}
    (g(x_-) - g(x_+))^2 &= \left(\sum_{S \subseteq [d]} \hat{g}(S) (x_-^S - x_+^S)\right)^2 \le \left(\sum_{\substack{S \subseteq [d],S \neq \emptyset}} |S| \hat{g}(S)^2\right) \left(\sum_{\substack{S \subseteq [d], S \neq \emptyset}} \frac{(x_-^S - x_+^S)^2}{|S|}\right) \\
    &\le \frac{1}{4} \Inf[g] \cdot 4 \sum_{\substack{S \subseteq [d] \\ S \neq \emptyset}} \frac{1}{|S|} = \Inf[g] \cdot \sum_{k = 1}^d \frac{1}{k} \binom{d}{k}\,. 
    \end{aligned} \end{equation}
    Finally, since $\tfrac{1}{k+1} \binom{d}{k} = \frac{1}{d+1} \binom{d+1}{k+1}$, we have
    \begin{equation} \label{eq:sb-i-3}
    \sum_{k = 1}^d \frac{1}{k} \binom{d}{k} \le 2 \sum_{k = 1}^d \frac{1}{k+1} \binom{d}{k} = \frac{2}{d+1} \sum_{k = 1}^d \binom{d+1}{k+1} \le \frac{4 \cdot 2^d}{d+1}\,.
    \end{equation}
    Together, \cref{eq:sb-i-1,eq:sb-i-2,eq:sb-i-3} imply \cref{eq:sb-i-0}, completing the proof.
\end{proof}

\begin{proof}[\textbf{Proof of \cref{thm:small-ball-anticoncentration}}]
    By rescaling, we may assume that $\gamma = 1$. By the definition of $r = r_{1}(F)$, there exist pairwise disjoint sets $S_1, \ldots, S_r \subseteq [n]$ of size $d$ such that $|a_{S_i}| \ge 1$ for all $i \in [r]$. By conditioning on the outcomes of $\xi_j$ for all $j \notin S_1 \cup \ldots \cup S_r$, we may also assume that $r = n/d$.  
    
    Our goal is to bound the probability that $|F(\xi_1, \ldots, \xi_n)| \le 1$. In fact, it is sufficient to show that
    \begin{equation} \label{eq:smaller-ball}
    \PP\big[|F(\xi_1, \ldots, \xi_n)| \le \tfrac{1}{2}\big] = O\left(\frac{2^d d}{\sqrt{n}}\right).
    \end{equation}
    Indeed, one can cover the disk of radius $1$ by a constant number of disks of radius $\tfrac{1}{2}$ (say, centered at the points $z_1, \ldots, z_{M} \in \CC$), and then apply \cref{eq:smaller-ball} to the polynomials $F - z_1, \ldots, F - z_{M}$.

    Denote $f = \frac{1}{1 + |F|^2}$. Since $|F|^2 = F \cdot \overline{F}$ is a polynomial of degree at most $2d$ with real coefficients, we can apply \cref{thm:influence} to $f$ and obtain that 
    \begin{equation} \label{eq:small-ball-1}
    \Inf[f] = O\left(d \sqrt{n}\right).
    \end{equation}
    For each $i \in [r]$, let $F_i$ be the (random) polynomial in the variables $(x_j)_{j \in S_i}$ obtained from $F$ by fixing the values of $x_j$ to $\xi_j$ for all $j \notin S_i$, and denote $f_i = \frac{1}{1 + |F_i|^2}$. Since $|a_{S_i}| \ge 1$ and the coefficient of $x^{S_i}$ in $F_i$ is exactly $a_{S_i}$, \cref{lemma:anticoncentration-influence} yields
    \begin{equation} \label{eq:small-ball-2}
    \PP\big[|F(\xi_1, \ldots, \xi_n)| \le \tfrac{1}{2}\big] = \EE_i\Big[\PP\big[|F_i((\xi_j)_{j \in S_i})| \le \tfrac{1}{2}\big]\Big] = O\left(\frac{2^d}{d} \EE_i\big[\Inf[f_i]\big]\right),
    \end{equation}
    where $\EE_i$ denotes the expectation over the random variables $(\xi_j)_{j \notin S_i}$. Finally, from the definition of total influence, we have $\Inf[f] = \sum_{i=1}^r \EE_i[\Inf[f_i]]$ with $r=n/d$. Together with \cref{eq:small-ball-1} and \cref{eq:small-ball-2}, this implies that
    \begin{equation*}
        \frac{n}{d} \cdot \PP\big[|F(\xi_1, \ldots, \xi_n)| \le \tfrac{1}{2}\big] = O\left(2^d \sqrt{n}\right),
    \end{equation*}
    which is exactly the desired bound \cref{eq:smaller-ball}.
\end{proof}

\begin{proof}[\textbf{Proof of \cref{thm:pointwise-anticoncentration}}]
    Denote $\gamma = \min\{|a_S|: S \subseteq [n],\; |S| = d,\; a_S \neq 0\}$. Then, $r_{\gamma}(F) = r(F)$, and by \cref{thm:small-ball-anticoncentration},
    \begin{equation*}
        \PP\big[F(\xi_1, \ldots, \xi_n) = 0\big] \le \PP\big[|F(\xi_1, \ldots, \xi_n)| \le \gamma\big] \le \frac{C 2^d \sqrt{d}}{\sqrt{r_{\gamma}(F)}} = \frac{C 2^d \sqrt{d}}{\sqrt{r(F)}}\,. \qedhere 
    \end{equation*}
\end{proof}

\begin{proof}[\textbf{Proof of \cref{thm:pointwise-anticoncentration-robust}}]
    We argue by induction on $d$ that the desired bound holds with $C_1 = \sqrt{2} C$, where $C$ is the constant in \cref{thm:pointwise-anticoncentration}. If $r(F) \ge \frac{1}{2d} s(F)$, then this bound follows from \cref{thm:pointwise-anticoncentration}. Otherwise, let $S_1, \ldots, S_{r(F)}$ be pairwise disjoint sets of size $d$ such that $a_{S_i} \neq 0$ for all $i$. Then, by the definition of $r(F)$, $a_S = 0$ for all sets $S \subseteq [n]$ of size $d$ that do not intersect the set $I = S_1 \cup \ldots \cup S_{r(F)}$. Condition on an arbitrary outcome of the random variables $(\xi_i)_{i \in I}$, and note that the resulting polynomial $F_0$ in the remaining variables $(x_i)_{i \notin I}$ has degree at most $d-1$. It suffices to show that the probability (with respect to the randomness of $(\xi_i)_{i \notin I}$) that $F_0((\xi_i)_{i \notin I}) = 0$ is at most $\frac{C_1 2^d d}{\sqrt{s(F)}}$. We distinguish the following two cases:

    \textbf{Case 1:} $F_0$ is a constant polynomial. Since $F_0$ is obtained from $F$ by fixing the values of $|I| = d \cdot r(F) < \frac{1}{2}s(F) < s(F)$ variables, this constant is not zero, and thus 
    \begin{equation*}
    \PP\big[F_0((\xi_i)_{i \notin I}) = 0\big] = 0\,.
    \end{equation*}
    
    \textbf{Case 2:} $1 \le \deg(F_0) \le d-1$. Again, since $F_0$ is obtained from $F$ by fixing the values of $|I|$ variables, we have $s(F_0) \ge s(F) - |I| > \frac{1}{2}s(F)$. Therefore, by the induction hypothesis,
    \begin{equation*}
    \PP\big[F_0((\xi_i)_{i \notin I}) = 0\big] \le \frac{C_1 2^{d-1} (d-1)}{\sqrt{s(F_0)}} \le \frac{C_1 2^d d}{\sqrt{s(F)}}\,. \qedhere
    \end{equation*}
\end{proof}

\bibliographystyle{plain}
\bibliography{references.bib}

@misc {jin-kwan-sauermann-wang,
  title        = {Algebraic aspects of the polynomial {Littlewood}--{Offord} problem},
  author       = {Zhihan Jin and Matthew Kwan and Lisa Sauermann and Yiting Wang},
  note         = {Preprint, arXiv:2505.23335},
  year         = {2025},
  url          = {https://arxiv.org/abs/2505.23335}
}

@article {kwan-sauermann,
    AUTHOR = {Kwan, Matthew and Sauermann, Lisa},
     TITLE = {Resolution of the quadratic {L}ittlewood-{O}fford problem},
   SJOURNAL = {Compos. Math.},
  JOURNAL = {Compositio Mathematica},
    VOLUME = {161},
      YEAR = {2025},
    NUMBER = {12},
     PAGES = {3089--3139},
      ISSN = {0010-437X,1570-5846},
   MRCLASS = {60C05},
  MRNUMBER = {5051290},
       DOI = {10.1112/S0010437X25102789},
       URL = {https://doi.org/10.1112/S0010437X25102789},
}

@incollection {nguyen-survey,
    AUTHOR = {Nguyen, Hoi H. and Vu, Van H.},
     TITLE = {Small ball probability, inverse theorems, and applications},
 BOOKTITLE = {Erd\H{o}s centennial},
    SERIES = {Bolyai Soc. Math. Stud.},
    VOLUME = {25},
     PAGES = {409--463},
 PUBLISHER = {J\'{a}nos Bolyai Math. Soc., Budapest},
      YEAR = {2013},
   MRCLASS = {60E15 (60C05 60F10)},
  MRNUMBER = {3203607},
MRREVIEWER = {Sho Matsumoto},
       DOI = {10.1007/978-3-642-39286-3\_16},
       URL = {https://doi.org/10.1007/978-3-642-39286-3_16},
}

@article{costello-tao-vu,
    AUTHOR = {Costello, Kevin P. and Tao, Terence and Vu, Van},
     TITLE = {Random symmetric matrices are almost surely nonsingular},
   SJOURNAL = {Duke Math. J.},
  JOURNAL = {Duke Mathematical Journal},
    VOLUME = {135},
      YEAR = {2006},
    NUMBER = {2},
     PAGES = {395--413},
      ISSN = {0012-7094,1547-7398},
   MRCLASS = {15A52 (05C35 05D40 60C05)},
  MRNUMBER = {2267289},
MRREVIEWER = {Steven\ Joel\ Miller},
       DOI = {10.1215/S0012-7094-06-13527-5},
       URL = {https://doi.org/10.1215/S0012-7094-06-13527-5},
}

@article {costello,
    AUTHOR = {Costello, Kevin P.},
     TITLE = {Bilinear and quadratic variants on the {L}ittlewood-{O}fford
              problem},
   SJOURNAL = {Israel J. Math.},
  JOURNAL = {Israel Journal of Mathematics},
    VOLUME = {194},
      YEAR = {2013},
    NUMBER = {1},
     PAGES = {359--394},
      ISSN = {0021-2172,1565-8511},
   MRCLASS = {60G50 (11K65 60E15)},
  MRNUMBER = {3047075},
MRREVIEWER = {Eugenijus\ Manstavi\v cius},
       DOI = {10.1007/s11856-012-0082-4},
       URL = {https://doi.org/10.1007/s11856-012-0082-4},
}

@article {erdos-littlewood-offord,
    AUTHOR = {Erd{\H o}s, P.},
     TITLE = {On a lemma of {L}ittlewood and {O}fford},
   SJOURNAL = {Bull. Amer. Math. Soc.},
  JOURNAL = {Bulletin of the American Mathematical Society},
    VOLUME = {51},
      YEAR = {1945},
     PAGES = {898--902},
      ISSN = {0002-9904},
   MRCLASS = {60.0X},
  MRNUMBER = {14608},
MRREVIEWER = {M.\ Kac},
       DOI = {10.1090/S0002-9904-1945-08454-7},
       URL = {https://doi.org/10.1090/S0002-9904-1945-08454-7},
}

@article {littlewood-offord,
    AUTHOR = {Littlewood, J. E. and Offord, A. C.},
     TITLE = {On the number of real roots of a random algebraic equation.
              {III}},
   SJOURNAL = {Rec. Math. [Mat. Sbornik] N.S.},
  JOURNAL = {Rec. Math. [Mat. Sbornik] N.S.},
    VOLUME = {12/54},
      YEAR = {1943},
     PAGES = {277--286},
   MRCLASS = {41.1X},
  MRNUMBER = {9656},
MRREVIEWER = {A.\ Zygmund},
}

@article{meka-nguyen-vu,
 author = {Meka, Raghu and Nguyen, Oanh and Vu, Van},
 title = {Anti-concentration for Polynomials of Independent Random Variables},
 year = {2016},
 pages = {1--17},
 doi = {10.4086/toc.2016.v012a011},
 publisher = {Theory of Computing},
 journal = {Theory of Computing},
 volume = {12},
 number = {11},
 URL = {https://theoryofcomputing.org/articles/v012a011},
}

@article {kane,
    AUTHOR = {Kane, Daniel M.},
     TITLE = {The correct exponent for the {G}otsman-{L}inial conjecture},
   SJOURNAL = {Comput. Complexity},
  JOURNAL = {Computational Complexity},
    VOLUME = {23},
      YEAR = {2014},
    NUMBER = {2},
     PAGES = {151--175},
      ISSN = {1016-3328,1420-8954},
   MRCLASS = {68Q15 (68Q87)},
  MRNUMBER = {3212596},
       DOI = {10.1007/s00037-014-0086-z},
       URL = {https://doi.org/10.1007/s00037-014-0086-z},
}

@article {jain-kwan-mubayi-tran,
    AUTHOR = {Jain, Vishesh and Kwan, Matthew and Mubayi, Dhruv and Tran,
              Tuan},
     TITLE = {The edge-statistics conjecture for hypergraphs},
   SJOURNAL = {Int. Math. Res. Not. IMRN},
  JOURNAL = {International Mathematics Research Notices. IMRN},
      YEAR = {2025},
    NUMBER = {18},
     PAGES = {Paper No. rnaf273, 19 pp.},
      ISSN = {1073-7928,1687-0247},
   MRCLASS = {05C35 (05C65)},
  MRNUMBER = {4961240},
MRREVIEWER = {Dan\ Hefetz},
       DOI = {10.1093/imrn/rnaf273},
       URL = {https://doi.org/10.1093/imrn/rnaf273},
}

@article {kwan-sudakov-tran,
    AUTHOR = {Kwan, Matthew and Sudakov, Benny and Tran, Tuan},
     TITLE = {Anticoncentration for subgraph statistics},
   SJOURNAL = {J. Lond. Math. Soc. (2)},
  JOURNAL = {Journal of the London Mathematical Society},
    VOLUME = {99},
      YEAR = {2019},
    NUMBER = {3},
     PAGES = {757--777},
      ISSN = {0024-6107,1469-7750},
   MRCLASS = {05C35 (05D40 60C05)},
  MRNUMBER = {3977889},
MRREVIEWER = {Yuval\ Filmus},
       DOI = {10.1112/jlms.12192},
       URL = {https://doi.org/10.1112/jlms.12192},
}

@misc{grebennikov-kwan-edge-statistics,
      title={Edge-statistics beyond $1/e$}, 
      author={Alexandr Grebennikov and Matthew Kwan},
      year={2025},
      eprint={2510.24691},
      archivePrefix={arXiv},
      primaryClass={math.CO},
      url={https://arxiv.org/abs/2510.24691}, 
      note={Preprint, arXiv:2510.24691},
}

@misc{grebennikov-kwan-littlewood-offord,
      title={Geometric {L}ittlewood-{O}fford problems via lattice point counting}, 
      author={Alexandr Grebennikov and Matthew Kwan},
      year={2025},
      eprint={2505.24699},
      archivePrefix={arXiv},
      primaryClass={math.CO},
      url={https://arxiv.org/abs/2505.24699},
      note={Preprint, arXiv:2505.24699}, 
}

@article {razborov-viola,
    AUTHOR = {Razborov, Alexander and Viola, Emanuele},
     TITLE = {Real advantage},
   SJOURNAL = {ACM Trans. Comput. Theory},
  JOURNAL = {ACM Transactions on Computation Theory},
    VOLUME = {5},
      YEAR = {2013},
    NUMBER = {4},
     PAGES = {Art. 17, 8 pp.},
      ISSN = {1942-3454,1942-3462},
   MRCLASS = {68Q05 (06E30 68Q15)},
  MRNUMBER = {3146710},
MRREVIEWER = {G\'abor\ Horv\'ath},
       DOI = {10.1145/2540089},
       URL = {https://doi.org/10.1145/2540089},
}

@article{rosinski-samorodnitsky,
 ISSN = {00911798, 2168894X},
 URL = {http://www.jstor.org/stable/2244849},
 author = {Jan Rosiński and Gennady Samorodnitsky},
 journal = {The Annals of Probability},
 number = {1},
 pages = {422--437},
 publisher = {Institute of Mathematical Statistics},
 title = {Symmetrization and Concentration Inequalities for Multilinear Forms with Applications to Zero-One Laws for {L}évy Chaos},
 urldate = {2026-08-04},
 volume = {24},
 year = {1996}
}

@misc{KKWY,
  title         = {Rational degree is polynomially related to degree},
  author        = {Kothari, Robin and Kovacs-Deak, Matt and Wang, Daochen and Yang, Rain Zimin},
  year          = {2026},
  eprint        = {2601.08727},
  note          = {Preprint, arXiv:2601.08727},
}

@article {gotsman-linial,
    AUTHOR = {Gotsman, Craig and Linial, Nathan},
     TITLE = {Spectral properties of threshold functions},
  SJOURNAL = {Combinatorica},
   JOURNAL = {Combinatorica},
    VOLUME = {14},
      YEAR = {1994},
    NUMBER = {1},
     PAGES = {35--50},
      ISSN = {0209-9683},
       DOI = {10.1007/BF01305949},
       URL = {https://doi.org/10.1007/BF01305949},
}

@inproceedings {chapman,
    AUTHOR = {Chapman, Brynmor},
     TITLE = {The {G}otsman-{L}inial conjecture is false},
 BOOKTITLE = {Proceedings of the 2018 {A}nnual {ACM}-{SIAM} {S}ymposium on
              {D}iscrete {A}lgorithms ({SODA})},
     PAGES = {692--699},
 PUBLISHER = {SIAM, Philadelphia, PA},
      YEAR = {2018},
      ISBN = {978-1-61197-503-1},
       DOI = {10.1137/1.9781611975031.45},
       URL = {https://doi.org/10.1137/1.9781611975031.45},
}

@misc {kim-maldonado-wellens,
      title={On graphs and the {G}otsman-{L}inial conjecture for $d=2$},
      author={Hyo Won Kim and Chris Maldonado and Jake Wellens},
      year={2017},
      eprint={1709.06650},
      archivePrefix={arXiv},
      primaryClass={math.CO},
      url={https://arxiv.org/abs/1709.06650},
      note={Preprint, arXiv:1709.06650},
}

@book{odonnell-book,
  author    = {Ryan O'Donnell},
  title     = {Analysis of Boolean Functions},
  publisher = {Cambridge University Press},
  year      = {2014}
}

@misc{chang-liu-liu,
      title={A proof of {Chv\'atal}'s conjecture via a sharp correlation inequality},
      author={Fan Chang and Hong Liu and Miao Liu},
      year={2026},
      eprint={2609.19123},
      archivePrefix={arXiv},
      primaryClass={math.CO},
      url={https://arxiv.org/abs/2609.19123},
      note={Preprint, arXiv:2609.19123},
}

@misc{ellis-filmus-friedgut,
      title={{Chv\'atal}'s conjecture: a proof from {The Book}},
      author={David Ellis and Yuval Filmus and Ehud Friedgut},
      year={2026},
      eprint={2609.28404},
      archivePrefix={arXiv},
      primaryClass={math.CO},
      url={https://arxiv.org/abs/2609.28404},
      note={Preprint, arXiv:2609.28404},
}

@misc{keevash,
      title={On {K}ahn's flow conjecture}, 
      author={Peter Keevash},
      year={2026},
      eprint={2609.23595},
      archivePrefix={arXiv},
      primaryClass={math.CO},
      url={https://arxiv.org/abs/2609.23595}, 
      note={Preprint, arXiv:2609.23595},
}

@inproceedings {ABKRT,
    AUTHOR = {Aaronson, Scott and Ben-David, Shalev and Kothari, Robin and
              Rao, Shravas and Tal, Avishay},
     TITLE = {Degree vs. approximate degree and quantum implications of
              {H}uang's sensitivity theorem},
 BOOKTITLE = {Proceedings of the 53rd {A}nnual {ACM} {SIGACT} {S}ymposium on
              {T}heory of {C}omputing ({STOC})},
     PAGES = {1330--1342},
 PUBLISHER = {ACM, New York},
      YEAR = {2021},
      ISBN = {978-1-4503-8053-9},
       DOI = {10.1145/3406325.3451047},
       URL = {https://doi.org/10.1145/3406325.3451047},
}

@misc{openai2026-gotsman-linial,
  author = {{OpenAI}},
  title  = {Average sensitivity of polynomial threshold functions},
  year   = {2026},
  note   = {\url{https://github.com/openai/math/blob/main/preprints/Average-Sensitivity-of-Polynomial-Threshold-Functions-September-25-2026/main.pdf}}
}

\end{document}